\documentclass[10pt]{amsart}
 \usepackage[utf8]{inputenc} 
\usepackage{csquotes}

\newtheorem{thm}{Theorem}[section]
\newtheorem{cor}[thm]{Corollary}
\newtheorem{lem}[thm]{Lemma}
\newtheorem{prop}[thm]{Proposition}

\theoremstyle{definition}

\theoremstyle{remark}
\newtheorem{rem}[thm]{Remark}

\numberwithin{equation}{section}
\newtheorem{quest}[thm]{Question}

\newcommand{\rank}{\operatorname{\rm rank}}

\begin{document}

\title[Combinatorial Formula for Operator Sequences]{A Combinatorial Formula for Recursive \\ Operator Sequences and Applications}

\author[R.E. Curto]{Ra\'ul E. Curto}
\address{Department of Mathematics, The University of Iowa, Iowa City, Iowa, U.S.A.}
\email{raul-curto@uiowa.edu}

\author[A. Ech-charyfy]{Abderrazzak Ech-charyfy}
\address{Laboratory of Mathematical Analysis and Applications,\\
Faculty of Sciences, Mohammed V University in Rabat,\\
Rabat, Morocco}
\email{abderrazzak.echcharyfy@gmail.com and abderrazzak\_echcharyfy@um5.ac.ma}

\author[K. Idrissi]{Kaissar Idrissi}
\address{Laboratory of Mathematical Analysis and Applications,\\
Faculty of Sciences, Mohammed V University in Rabat,\\
Rabat, Morocco}
\email{kaissar.idrissi@fsr.um5.ac.ma}

\author[E.H. Zerouali]{El Hassan Zerouali}
\address{Laboratory of Mathematical Analysis and Applications,\\
Faculty of Sciences, Mohammed V University in Rabat,\\
Rabat, Morocco}
\email{elhassan.zerouali@fsr.um5.ac.ma}


\begin{abstract} \ We study sequences of bounded operators \((T_n)_{n \ge 0}\) on a complex separable Hilbert space \(\mathcal{H}\) that satisfy a linear recurrence relation of the form
$$
T_{n+r} = A_0 T_n + A_1 T_{n+1} + \cdots + A_{r-1} T_{n+r-1} \quad(\textrm{for all } n\ge 0),
$$
where the coefficients \(A_0, A_1, \dots, A_{r-1}\) are pairwise commuting bounded operators on \(\mathcal{H}\). \ 
Such relations naturally arise in the context of the operator-valued moment problem, particularly in the study of flat extensions of block Hankel operators. \ Our first goal is to derive an explicit combinatorial formula for   \(T_n\).
 As a concrete application, we provide an explicit expression for the powers of an operator-valued companion matrix. \ In the special case of scalar coefficients $A_k=a_kI_\mathcal{H}$, with $a_k\in\mathbb{R}$, we recover a Binet-type formula that allows the explicit computation of the powers and the exponential of algebraic operators in terms of Bell polynomials.
\end{abstract} 

\vspace{5pt}

		%
%
\maketitle
%


\section{Introduction}\label{sec:introduction}

Let $s^{(r)} = \{s_k\}_{k=0}^r$ be a finite sequence of real numbers. \ Tchakaloff's Theorem~\cite{Tchakaloff} states that if there exists a positive Borel measure $\mu$ on $\mathbb{R}$ such that
\(
s_k = \displaystyle\int_{\mathbb{R}} x^k\, d\mu(x),  
\)
then there exists a positive finitely atomic measure 
\(
\nu = \displaystyle\sum_{i=1}^m w_i\, \delta_{x_i},
\)
such that
\begin{equation}\label{recr}
s_k = \displaystyle \int_{\mathbb{R}} x^k\, d\nu(x)=\sum_{i=1}^m w_i x_i^k, \text{ for } k = 0, 1, \dots, r.    
\end{equation}
The expression in Equation \eqref{recr} allows us to extend our sequence in a moment sequence satisfying a recursive relation given by \begin{equation}\label{recur}
    s_{m+k} = \sum_{j=0}^{m-1} a_j s_{k+j}, \text{ for all } k \geq 0.
\end{equation}
associated with  the monic polynomial
$ P(X) = \displaystyle\prod_{i=1}^m (X - x_i) = X^m - \sum_{j=0}^{m-1} a_j X^j.
$

Sequences $\{s_k\}_{k \ge 0}$ satisfying a linear recurrence relation \eqref{recur} 
have been widely studied in different mathematical branches because of the large range of applications. \ They are known as  \emph{generalized Fibonacci sequences} in combinatorics and discrete mathematics, as \emph{linear difference equations} in numerical analysis, and as \emph{moment recursive sequences} in the context of the scalar moment problem. \ Hence, Tchakaloff's Theorem can be reformulated as follows: a finite sequence \(s^{(r)}\) is the moment sequence of some nonnegative measure if and only if it can be extended to a full moment recursive sequence. \  
As a consequence, the associated Hankel matrices satisfy the  \emph{flat extension property}; namely,
$$
\rank (s_{i+j+k})_{0 \le i,j \le r} 
= \rank (s_{i+j+k+1})_{0 \le i,j \le r+1}, \ \text{for any } k \ge 0. 
$$
This establishes the equivalence of three key properties~\cite{curto1991recursiveness, CF96, CF98}:
\begin{enumerate}
  \item The existence of a finitely atomic representing measure;  
  \item the recursiveness of the extended moment sequence; and 
  \item the flat extension property.
\end{enumerate}

In the \emph{matrix-valued case}, where the scalars in \eqref{recur} are replaced by Hermitian matrices, this equivalence has also been investigated, and the previous observations have been recovered. \ In particular, the authors in~\cite{Kimsey2014, KT} extended Tchakaloff's Theorem and the flatness property to truncated sequences of Hermitian matrices. \  
The connection between the scalar and matrix cases has been further explored in~\cite{cehz}, where it is shown that the union of the supports of the scalar representing measures reflects the support of the matrix-valued measure.

In the \emph{operator-valued setting}, scalar moments are replaced by self-adjoint operators acting on an infinite-dimensional Hilbert space \(\mathcal{H}\). \ It has been shown (see~\cite{curto2025local, Kimsey2014}) that Tchakaloff's Theorem does not extend to this context. \ More precisely, there exist finite families of self-adjoint operators that arise as moments of some nonnegative operator-valued measure, yet for which no finitely atomic operator-valued representing measure exists. \ Moreover, a truncated sequence of operator moments may fail to be recursively extendable in the operator moment problem. \  
This naturally raises the following question in the operator moment problem framework:
\begin{quest}
Are there suitable formulations of \emph{recursiveness} and \emph{flatness} under which the equivalence 
\((1) \iff (2) \iff (3)\) continues to hold, as in the finite-dimensional case?
\end{quest}  
In recent work \cite{flat}, the authors have extended the classical notion of \emph{flatness} from the finite-dimensional to the infinite-dimensional operator setting. \  
Specifically, consider a block operator
\[
X = 
\begin{pmatrix}
A & B \\
B^* & C
\end{pmatrix}
\]
acting on the Hilbert space \(\mathcal{H} \oplus \mathcal{K}\). \  
We say that \(X\) is a \emph{flat extension} of \(A\) if and only if
\[\mathcal{H} \oplus \mathcal{K}
= (\mathcal{H}, 0_{\mathcal{K}}) + \ker X.
\]
This definition preserves the essential geometric structure of flatness and naturally generalizes the finite-dimensional concept (see~\cite[Proposition 2.1]{flat in algebra}).  
It provides a robust framework for extending moment-theoretic and algebraic ideas to operator-valued contexts, where positivity and recursive properties coexist in infinite dimensions. As an application, we consider a sequence \((T_n)_{n\ge0}\) of bounded self-adjoint operators on a Hilbert space \(\mathcal{H}\).  
The block Hankel operator
\(
H_n = (T_{i+j})_{0 \le i,j \le n}
\)
acting on the Hilbert space
\(
\mathcal{H}^{(n)} = \underbrace{\mathcal{H} \oplus \dots \oplus \mathcal{H}}_{n+1 \text{ times}}
\)
admits a \emph{flat extension}
\(
H_{n+1} = (T_{i+j})_{0 \le i,j \le n+1}
\)
if and only if
\(
\mathcal{H}^{(n+1)} = (\mathcal{H}^{(n)}, 0_{\mathcal{H}}) + \ker H_{n+1}.
\)
Under this formulation, the flatness of the Hankel operator encodes the recursive structure of the sequence \((T_n)_{n\ge0}\). \  
Indeed, one can show that a truncated sequence admitting a flat extension determines a sequence of bounded self-adjoint operators satisfying a linear operator recurrence relation of order \(r\):
\begin{equation}\label{recurrence}
T_{n+r} = A_0 T_n + A_1 T_{n+1} + \cdots + A_{r-1} T_{n+r-1}, \quad \forall n \ge 0,
\end{equation}
where the coefficients
\(
\mathbf{A} = (A_0, A_1, \dots, A_{r-1})
\)
are bounded self-adjoint operators on \(\mathcal{H}\) with \(A_{r-1} \neq 0_{\mathcal{H}}\).

Our objective in this work is to study and classify the types of sequences that satisfy 
a linear recurrence relation of the form~\eqref{recurrence} in the case where the 
$r$-tuple $\mathbf{A} = (A_0, A_1, \dots, A_{r-1})$ consists of pairwise commuting self-adjoint operators.

\bigskip
\subsection{Our approach}
Our methodology can be outlined as follows:
\begin{itemize}
    \item We extend the recurrence relation~\eqref{recurrence} to a more general vectorial framework. 
    \item We rely on the spectral theory of commuting $r$-tuples of self-adjoint operators.
By means of the joint spectral theorem, the vectorial framework is reduced to a family of scalar recurrence problems on a measurable space. 
    \item For the scalar setting, we recover classical results, which are then lifted back to the operator framework 
    by means of continuous functional calculus. \ 
    This spectral approach enables us to construct explicit formulas for operator sequences satisfying~\eqref{recurrence}.
\end{itemize}

\subsection{Main contributions}
The main contributions of this paper can be summarized as follows:
\begin{itemize}
    \item We establish an operator-valued analogue of the combinatorial formula for 
    linear recurrences \eqref{recurrence}, valid in the context of commuting self-adjoint operator coefficients.
    \item As a concrete application, we provide an explicit formula for the powers of the operator companion matrix, 
    obtained by reformulating the recurrence as a first-order dynamical system.
    \item In the case where the coefficients are scalar, we obtain a Binet-type formula, which allows us to explicitly 
    characterize both the powers and the exponential of an algebraic operator in terms of Bell polynomials.
\end{itemize}

\section{Preliminaries}
\subsection{Notation and terminology}

For an integer \( r \geq 2 \), a \emph{multi-index} of length \( r \) is a vector
\[
\mathbf{k} := (k_0, k_1, \dots, k_{r-1}) \in \mathbb{Z}_+^r.
\]
Its \emph{length} is defined by
\[
|\mathbf{k}| := k_0 + k_1 + \cdots + k_{r-1},
\]
and its \emph{weighted degree}, i.e., the scalar product with the vector 
\(\mathbf{d} := (1,2,\dots,r)\), is given by
\[
\langle \mathbf{k}, \mathbf{d} \rangle := \sum_{j=0}^{r-1} (j+1) k_j 
= k_0 + 2k_1 + \cdots + r k_{r-1}.
\]
The \emph{multinomial coefficient} associated with 
\(\mathbf{k} \in \mathbb{Z}_+^r\) is
\[
\binom{|\mathbf{k}|}{\mathbf{k}} := \frac{(|\mathbf{k}|)!}{k_0! \cdots k_{r-1}!},
\]
while the \emph{multivariate monomial} of index \(\mathbf{k}\) associated with 
\(\mathbf{a} = (a_0, a_1, \dots, a_{r-1}) \in \mathbb{C}^r\) is defined as
\[
\mathbf{a}^{\mathbf{k}} := a_0^{k_0} a_1^{k_1} \cdots a_{r-1}^{k_{r-1}}.
\]

Throughout, \( \mathbf{B}(\mathcal{H}) \) denotes the algebra of bounded linear operators on a separable complex Hilbert space \( \mathcal{H} \); 
\( I_{\mathcal{H}} \) and \( 0_{\mathcal{H}} \) denote, respectively, the identity and the zero operator on \( \mathcal{H} \). \ 
We also write \( \mathbf{M}_d(\mathbb{C}) \) for the algebra of \( d \times d \) complex matrices, with \( I_d \) and \( 0_d \) denoting, respectively, the identity and the zero matrix in \( \mathbb{C}^d \).

\bigskip
\subsection{Some known results}

Let \((\gamma_n)_{n \geq 0}\) be a sequence of complex numbers determined by the initial conditions
\(
\gamma_0 = \alpha_0, \quad \gamma_1 = \alpha_1, \quad \dots, \quad \gamma_{r-1} = \alpha_{r-1},
\)
and satisfying the linear recurrence relation of order \(r \geq 2\):
\begin{equation}\label{eq:rec_scalar}
\gamma_{n+r} = a_{r-1} \gamma_{n+r-1} + a_{r-2} \gamma_{n+r-2} + \cdots + a_0 \gamma_n, 
\quad \forall n \geq 0,
\end{equation}
where \(a_i \in \mathbb{C}\) are fixed coefficients with \(a_{r-1} \neq 0\).  
\\Sequences defined by \eqref{eq:rec_scalar}, usually referred to as \emph{\(r\)-generalized Fibonacci sequences}, have been extensively studied in the literature. \ They exhibit a rich algebraic and analytic structure, with connections to the theory of characteristic polynomials, generating functions, special functions, and also to moment problems in the scalar setting (see, for example,~\cite{taher2001recursive, chidume2001fib,levesque1985,mouline1999,philippou1988}). \  The recurrence relation \eqref{eq:rec_scalar} is governed by the characteristic polynomial
\[
P(X) = X^r - a_{r-1} X^{r-1} - a_{r-2} X^{r-2} - \cdots - a_0,
\]
whose spectral properties completely determine the behavior of the sequence. \ In particular, the sequence admits a closed-form representation of Binet type:
\[
\gamma_n = \sum_{i=1}^s \left( \sum_{j=0}^{m_i - 1} \beta_{i,j} n^j \right) \lambda_i^n,
\]
where \(\lambda_1,\dots,\lambda_s\) denote the distinct roots of \(P\), \(m_i\) their respective multiplicities, and the coefficients \(\beta_{i,j}\) are uniquely determined by the initial conditions.  
\\An alternative description of the general term, involving combinatorial coefficients, is given in~\cite{levesque1985}. \ More precisely, for every \(n \geq r\),
\begin{equation}\label{eq:main}
\gamma_n = \rho(n, r) W_0 + \rho(n - 1, r) W_1 + \cdots + \rho(n - r + 1, r) W_{r-1},
\end{equation}
where
\[
W_s = a_{r-1} \gamma_s + a_{r-2} \gamma_{s+1} + \cdots + a_s \gamma_{r-1}, 
\quad 0 \leq s \leq r-1,
\]
and
\[
\rho(n, r) = \sum_{\substack{\mathbf{k} \in \mathbb{Z}_+^r \\ \langle \mathbf{k}, \mathbf{d} \rangle = n - r}} 
\binom{|\mathbf{k}|}{\mathbf{k}}\, \mathbf{a}^{\mathbf{k}},
\]
with the conventions \(\rho(r, r) = 1\) and \(\rho(n, r) = 0\) for \(n < r\). \  The relationship between the combinatorial coefficients in~\eqref{eq:main} and the spectral data of the characteristic polynomial \(P\) was further analyzed in~\cite{taher2016binet}, where refined connections with binomial identities and generalizations of Binet’s formula are established. \ These representations provide useful tools in number theory, combinatorics, and the analysis of recursive algorithms.

In the matrix case, there is no direct analog of the Binet-type formula due to the 
non-commutativity of matrices and the lack of simultaneous diagonalizability. \ 
However, in~\cite{taher2001linear}, the authors generalized the combinatorial formula~\eqref{eq:main} 
to the algebra of matrices, relying on the fact that pairwise commuting symmetric matrices 
are simultaneously diagonalizable in the same basis. \ Using the Cayley--Hamilton Theorem, 
they derived, as an application, explicit formulas for \(A^n\) (\(n \geq r\)) and \(e^{tA}\) for every 
\(r \times r\) matrix \(A\), expressed in terms of the coefficients of its characteristic polynomial 
and the matrices \(A^j\), where \(0 \leq j \leq r - 1\). \  More specifically, consider a family 
\(\mathbf{A} = \{A_0, A_1, \dots, A_{r-1}\}\) of \(d \times d\) symmetric matrices 
that are pairwise commuting, with \(A_{r-1} \neq 0_d\). \  
Let \((Y_n)_{n \geq 0} \subseteq \mathbf{M}_d(\mathbb{C})\) be the matrix-valued sequence defined recursively by  
\begin{equation}\label{eq:Yn}
\begin{cases}
Y_i = V_i, & \text{for } i = 0,\dots, r-1, \\[6pt]
Y_n = A_0 Y_{n-1} + A_1 Y_{n-2} + \cdots + A_{r-1} Y_{n-r}, & \text{for } n \geq r,
\end{cases}
\end{equation}
where \(\{V_0, \dots, V_{r-1}\} \subseteq \mathbf{M}_d(\mathbb{C})\) is a prescribed set of initial matrices.  \\Then, for every \(n \geq r\), the sequence satisfies the recurrence relation
$$
Y_n = \displaystyle\sum_{s=0}^{r-1} \rho(n-s, r)\, W_s,
$$
where
$$
W_s = \displaystyle\sum_{j=s}^{r-1} A_j V_{s + r - 1 - j}, \quad (s = 0, 1, \dots, r-1),
$$
and the matrix coefficients \(\rho(n,r)\) are given by
$$
\rho(r,r) = I_d, \quad \rho(p,r) = 0_d \quad (\text{for } p < r),
$$
and, for all \(n \geq r\),
$$
\rho(n,r) = \displaystyle\sum_{\substack{\mathbf{k} \in \mathbb{Z}_+^r \\ \langle \mathbf{k}, \mathbf{d} \rangle = n-r}}
\binom{|\mathbf{k}|}{\mathbf{k}}\, \mathbf{A}^\mathbf{k},
$$
with
$$
\mathbf{A}^\mathbf{k} = A_0^{k_0} A_1^{k_1} \cdots A_{r-1}^{k_{r-1}}.
$$
To extend the previous theorem to the infinite-dimensional case, we need some tools from spectral theory for commuting $n$-tuples of operators. \ Let $\mathbf{A} = (A_1, A_2, \dots, A_n)$ be a family of pairwise commuting self-adjoint operators on a Hilbert space $\mathcal{H}$. \ 
The following spectral theorem for commuting $n$-tuples of self-adjoint operators asserts:
\begin{thm}\cite[Theorem 5.23]{shmu}\label{j-spec}
There exists a unique spectral measure $E$ defined on the Borel 
$\sigma$-algebra $\mathcal{B}(\mathbb{R}^n)$ such that, for each $k = 1, \dots, n$,
\[
A_k = \int_{\sigma(\mathbf{A})} t_k \, dE(t_1, \dots, t_n),
\]
where $\sigma(\mathbf{A}) \subseteq \mathbb{R}^n$ is the joint spectrum of the $n$-tuple $\mathbf{A}$, i.e., the support of the spectral measure $E$. \ 
Moreover, this joint spectrum satisfies
\[
\sigma(\mathbf{A}) \subseteq \sigma(A_1) \times \cdots \times \sigma(A_n).
\] 
\end{thm}
A useful characterization of $n$-tuples of pairwise commuting operators can be viewed as an extension of the joint spectral theorem. \ This characterization parallels the classical single-operator case established by Halmos \cite{halmos}, which asserts that every cyclic self-adjoint operator is unitarily equivalent to a multiplication operator. \ The proof we present here follows the same strategy as in Halmos's approach.

\begin{thm}\label{uni} 
Let $\mathbf{A} = (A_1, A_2, \dots, A_n)$ be a family of pairwise commuting self-adjoint operators on $\mathcal{H}$, and assume that $\mathbf{A}$ admits a (joint) cyclic vector. \ Then, there exists a  measure space $(X, \mu)$, a unitary operator
\(
U : L^2(X, \mu) \to \mathcal{H},
\)
and measurable real-valued functions $a_k : X \to \mathbb{R}$ such that, for each $k = 1, \dots, n$,
\(
U^{-1} A_k U = M_{a_k},
\)
where $M_{a_k}$ is the multiplication operator defined by
\(
(M_{a_k} f)(x) = a_k(x) f(x), \text{ for all } f \in L^2(X, \mu) \text{ and almost every } x \in X.
\)
\end{thm}
\begin{proof}
Let $E$ denote the joint spectral measure associated with the $n$-tuple $(A_1,\dots,A_n)$, as \linebreak provided by Theorem~\ref{j-spec}.
For $\xi \in \mathcal{H}$, define a positive finite scalar measure $\mu_\xi$ on $\mathbb{R}^n$ by  
\[
\mu_\xi(\Delta) := \langle E(\Delta)\xi, \xi \rangle, 
\quad \Delta \in \mathcal{B}(\mathbb{R}^n).
\]  
Define a linear map on the space of complex polynomials in $n$ variables, $\mathcal{P}(\mathbb{R}^n)$, by
\[
U_0 : \mathcal{P}(\mathbb{R}^n) \longrightarrow \mathcal{H}, 
\qquad 
U_0(p) := p(A_1,\dots,A_n)\,\xi,
\]
For any polynomial $p\in \mathcal{P}(\mathbb{R}^n)$, we have
\[
\begin{array}{rcl}
\|U_0(p)\|^2 
&=& \langle p(A_1,\dots,A_n)\, \xi,\, p(A_1,\dots,A_n)\, \xi \rangle \\[2mm]
&=& \langle (\overline{p} p)(A_1,\dots,A_n)\, \xi,\, \xi \rangle \\[1mm]
&=& \displaystyle \int_{\mathbb{R}^n} |p(x_1,\dots,x_n)|^2 \, d\mu_\xi(x_1,\dots,x_n).
\end{array}
\]
Hence $U_0$ is an isometry if we identify the polynomial $p$ with the function $(x_1,\dots,x_n) \mapsto p(x_1,\dots,x_n)$ in $L^2(\mathbb{R}^n,\mu_\xi)$.
Polynomials in $n$ variables are dense in $L^2(\mathbb{R}^n, \mu_\xi)$ because continuous functions with compact support are dense, and they can be approximated by polynomials on the support of $\mu_\xi$. \ Thus the domain of $U_0$ is dense in $L^2(\mathbb{R}^n, \mu_\xi)$.
By continuity, $U_0$ extends to an isometry 
\(
U : L^2(\mathbb{R}^n, \mu_\xi) \longrightarrow \mathcal{H}.
\)

If $\xi$ is a cyclic vector for the $n$-tuple $(A_1,\dots,A_n)$, then the image of $U$ is the closure of 
\(
\{\, p(A_1,\dots,A_n)\, \xi : p \in \mathcal{P}(\mathbb{R}^n) \,\},
\)
which equals $\mathcal{H}$ by the cyclicity assumption. \ Therefore, $U$ is a unitary operator.  
\\For any polynomial $p \in \mathcal{P}(\mathbb{R}^n)$ and for each $j \in \{1,\dots,n\}$, we have
\[
\begin{aligned}
U^{-1} A_j U (p) 
&= U^{-1} \big( A_j \, (p(A_1,\dots,A_n)\, \xi) \big) \\
&= U^{-1} \big( (x_j \cdot p)(A_1,\dots,A_n)\, \xi \big) \\
&= x_j \cdot p.
\end{aligned}
\]
Since the polynomials form a dense subspace of $L^2(\mathbb{R}^n, \mu_\xi)$, this equality extends by continuity to all of $L^2(\mathbb{R}^n, \mu_\xi)$. \  
Thus, denoting by $M_{x_j}$ the multiplication operator by the $j$-th coordinate, we obtain
\[
U^{-1} A_j U = M_{x_j}, \text{ for every } j=1,\dots,n.
\]
\end{proof}

\begin{rem}
In the previous theorem, when no single cyclic vector exists, we construct a decomposition into cyclic subspaces. \ That is, there exists a family of closed subspaces $\{H_m\}_{m\in\mathbb{N}}$ such that
\[
\mathcal{H} = \bigoplus_{n \in \mathbb{N}} H_m,
\]
where each $H_m$ is reducing for the algebra $\mathcal{A} := C^*(A_1,\dots,A_n)$ and admits a cyclic vector $\xi_m \in H_m$.
\\Take a dense sequence $(e_k)_{k\ge1}$ in $\mathcal{H}$, since \(\mathcal{H}\) is separable Hilbert space. \ Inductively define
\[
H_1 := \overline{\mathcal{A} e_1}, \quad H_m := \overline{\mathcal{A} e_{k_m}}, \ m\ge2,
\]
where $e_{k_m}$ is the first vector not in $H_1\oplus\cdots\oplus H_{m-1}$. \ Each $H_m$ is nonzero, reducing, and cyclic, and the sum is dense:
\[
\mathcal{H} = \overline{\bigoplus_{m\in\mathbb{N}} H_m} = \bigoplus_{m\in\mathbb{N}} H_m.
\]
For each $H_m$ with cyclic vector $\xi_m$, define the scalar measure $\mu_m$ and unitary $U_m : L^2(\mathbb{R}^n, \mu_m) \to H_m$ as in the cyclic case. \  
Let $\mu := \displaystyle\sum_{m\in \mathbb{N}} \mu_m$ be the joint spectral measure. \ 
There is a natural isomorphism
\[
\bigoplus_{m\in\mathbb{N}} L^2(\mathbb{R}^n, \mu_m) \simeq L^2(\mathbb{R}^n, \mu),
\]
given by
\(
\Phi(f_1 \oplus f_2 \oplus \dots) :=\displaystyle \sum_{m\in\mathbb{N}} f_m,
\)
where the sum is well-defined $\mu$-almost everywhere because the measures $\mu_m$ are mutually singular on their supports. \ Define
\[
V : \bigoplus_{m \in \mathbb{N}} L^2(\mathbb{R}^n, \mu_m) \longrightarrow \mathcal{H}, \quad V(f_1 \oplus f_2 \oplus \dots) := \sum_{m\in \mathbb{N}} U_m f_m.
\]
Then the map
\[
U=V \circ \Phi^{-1} : L^2(\mathbb{R}^n, \mu) \longrightarrow \mathcal{H}, \,
U(f) := \sum_{m \in \mathbb{N}} U_m f_m, \ \text{with } \Phi^{-1}(f) = (f_m)_{m\in\mathbb{N}},
\]
is unitary. \ Moreover, for each $k = 1, \dots, n$,
\[
U^{-1} A_k U = M_{x_k} \quad \text{on } L^2(\mathbb{R}^n, \mu).
\]
\end{rem} 
\section{An Operator-valued Generalization of a Combinatorial Formula}
In the following, we consider a sequence of vectors \( (u_n)_{n\ge 0} \subseteq \mathcal{H} \) and a sequence of operators \( (T_n)_{n \ge 0} \subseteq \mathbf{B}(\mathcal{H}) \), satisfying, respectively, the following vector recurrence relation and the operator recurrence relation of order \( r \geq 2 \):
\begin{equation}\label{bin:vect}
u_{n+r} = A_0 u_n + A_1 u_{n+1} + \cdots + A_{r-1} u_{n+r-1}, \quad \forall n \geq 0,
\end{equation}
\begin{equation}\label{bin:ope}
T_{n+r} = A_0 T_n + A_1 T_{n+1} + \cdots + A_{r-1} T_{n+r-1}, \quad \forall n \geq 0,
\end{equation}
where the coefficients \( \mathbf{A} = (A_0, A_1, \dots, A_{r-1}) \) are bounded, self-adjoint, and pairwise commuting operators, with \( A_{r-1} \neq 0_{\mathcal{H}} \).

These recurrence relations play a central role in the study of discrete-time linear systems in both finite- and infinite-dimensional settings. \ 
The vector recurrence \eqref{bin:vect} generalizes classical linear recurrences for scalar and vector sequences, whereas the operator recurrence \eqref{bin:ope} provides a natural framework for analyzing sequences of operators in Hilbert spaces. 

We first consider the vector recurrence \eqref{bin:vect} as a preliminary step to study the operator recurrence \eqref{bin:ope}, since the latter is more general. \ Indeed, once a solution of \eqref{bin:vect} is known, the operator case can be addressed by evaluating each operator \(T_n\) on vectors \(h \in \mathcal{H}\), i.e.,
\(
u_n = T_n h.
\)
This observation reduces the study of the operator sequence to a family of vector sequences, which can then be analyzed using classical techniques.
By introducing the \emph{operator companion matrix} 
\begin{equation}\label{cm}
\mathbf{B} =
\begin{pmatrix}
A_{r-1} & A_{r-2} & \cdots & A_{1} & A_{0} \\
I_{\mathcal{H}} & 0 & \cdots & 0 & 0 \\
0 & I_{\mathcal{H}} & \ddots & \vdots & \vdots \\
\vdots & \ddots & \ddots & 0 & 0 \\
0 & \cdots & 0 & I_{\mathcal{H}} & 0
\end{pmatrix} \in \mathbf{B}(\mathcal{H}^{(r)}),
\end{equation}
we can rewrite the vector recurrence \eqref{bin:vect} as a first-order system in \(\mathcal{H}^{(r)}\):
\[
Y_{n+1} = \mathbf{B} Y_n, \quad Y_n := 
\begin{pmatrix} u_{n+r-1} \\ u_{n+r-2} \\ \vdots \\ u_n \end{pmatrix} \in \mathcal{H}^{(r)}.
\]
This formulation naturally expresses the evolution of the sequence through the powers of \(\mathbf{B}\):
\[
Y_n = \mathbf{B}^n Y_0, \quad Y_0 := 
\begin{pmatrix} u_{r-1} \\ u_{r-2} \\ \vdots \\ u_0 \end{pmatrix}.
\] 
Thus, the companion operator matrix \(\mathbf{B}\) provides a natural framework linking higher-order recurrences to first-order dynamical systems, while \(\mathbf{B}^n\) plays a central role in deriving explicit combinatorial and spectral formulas for both scalar and operator sequences \cite{benkhaldoun2021periodic,taher2003, chen1996companion}.

However, computing the powers of the companion operator matrix \(\mathbf{B}\) directly is highly nontrivial: there is no general algorithm for \(\mathbf{B}^n\) when the entries are noncommutative operators, especially in infinite-dimensional Hilbert spaces. \ Therefore, it is essential to develop a method to represent the vectors \(u_n\) explicitly in terms of the initial data \(u_0, \dots, u_{r-1}\). \ Once such a representation is established, the corresponding powers \(\mathbf{B}^n\) can be deduced indirectly, providing a systematic approach to the study of both the vector and operator recurrences. \ The following Remark illustrates a key aspect of our approach.
\begin{rem}
Let \( (u_n)_{n \ge 0} \subseteq \mathcal{H} \) be a sequence satisfying the operator-valued recurrence relation~\eqref{bin:vect}, and let 
\[
f_n := U^{-1} u_n \in L^2(X, \mu), \quad n \geq 0,
\]
as in Theorem~\ref{uni}. \ Then the sequence \( (f_n)_{n \ge 0} \subseteq L^2(X, \mu) \) satisfies the scalar recurrence relation
\begin{equation}\label{bin:fun}
f_{n+r}(x) = a_0(x)\, f_n(x) + a_1(x)\, f_{n+1}(x) + \cdots + a_{r-1}(x)\, f_{n+r-1}(x),
\end{equation}
for \(\mu\)-almost every \(x \in X\).
\end{rem}
From the scalar combinatorial formula~\eqref{eq:main}, we have the following lemma.

\begin{lem}\label{etu}
The general solution of~\eqref{bin:fun} satisfies, for every \(n \geq r\),
\[
f_n = \sum_{s=0}^{r-1} \rho(n-s, r ; \cdot)\, w_s, \quad \mu\text{-a.e.},
\]
where
\[
w_s := \sum_{j=s}^{r-1} a_j \, f_{\,s+r-1-j}, \quad s = 0, \dots, r-1,
\]
and the coefficients \(\rho(m,r)\) are given by the combinatorial formula
\[
\rho(n, r; \cdot) := 
\sum_{\substack{\mathbf{k} \in \mathbb{Z}_+^r \\ \langle \mathbf{k}, \mathbf{d} \rangle = n-r}}
\binom{|\mathbf{k}|}{\mathbf{k}} \,
a_0^{k_0} a_1^{k_1} \cdots a_{r-1}^{k_{r-1}},
\]
with the conventions
\[
\rho(r,r; \cdot) = 1, 
\qquad 
\rho(m,r,; \cdot) = 0 \quad \text{for } m < r.
\]
\end{lem}
We derive the following result.
\begin{thm}\label{mainr}
Let \((u_n)_{n \geq 0} \subseteq \mathcal{H}\) be a sequence satisfying the vector-valued recurrence relation \eqref{bin:vect}. \  
Then, for all \(n \geq r\), the following explicit expression holds:
\[
u_n = \sum_{s=0}^{r-1} \rho(n - s, r; \mathbf{A}) \, W_s,
\]
where:
\begin{itemize}
  \item the operator coefficients \(\rho(n, r; \mathbf{A})\) are given by
  \begin{equation}\label{exp1}
  \rho(m, r; \mathbf{A}) := \sum_{\substack{\mathbf{k} \in \mathbb{Z}_+^r \\ \langle \mathbf{k}, \mathbf{d} \rangle = m - r}} 
  \binom{|\mathbf{k}|}{\mathbf{k}} \, \mathbf{A}^{\mathbf{k}},
  \end{equation}
  with the multi-index notation \(\mathbf{A}^{\mathbf{k}} = A_0^{k_0} \cdots A_{r-1}^{k_{r-1}}\), and the conventions
  \begin{equation}\label{exp2}
  \rho(r, r; \mathbf{A}) = I_{\mathcal{H}} \quad \text{and} \quad \rho(m, r; \mathbf{A}) = 0_{\mathcal{H}} \quad \text{for } m < r.
  \end{equation}

  \item The vectors \(W_s \in \mathcal{H}\) are defined by
  \[
  W_s := \sum_{j=s}^{r-1} A_j \, u_{s + r - 1 - j}, \quad \text{for } s = 0, \dots, r-1.
  \]
\end{itemize}
\end{thm}

\begin{proof}
Under the assumption, the sequence $f_n=U^{-1}u_n$ satisfies  equation \eqref{bin:fun}.
From lemma \ref{etu}, we have for every \( n \geq r \):
\[
f_n = \sum_{s=0}^{r-1} \rho(n - s, r; \cdot) \, w_s, \quad \mu\text{-a.e },
\]
where
\[
w_s := \sum_{j=s}^{r-1} a_j \, f_{s + r - 1 - j},  \quad s = 0, \ldots, r-1,
\]
and the coefficients \(\rho(n, r; \cdot)\) are defined by the combinatorial formula:
\[
\rho(n, r; \cdot) := \sum_{\substack{\mathbf{k} \in \mathbb{Z}_+^r \\ \langle \mathbf{k}, \mathbf{d} \rangle = n - r}} 
\binom{|\mathbf{k}|}{\mathbf{k}} a_0^{k_0} a_1^{k_1} \cdots a_{r-1}^{k_{r-1}},
\]
with the conventions
\[
\rho(r, r; \cdot) = 1, \quad \text{and} \quad \rho(m, r; \cdot) = 0 \quad \text{for } m < r.
\]
Now, define the vectors in \(\mathcal{H}\):
\[
W_s := \sum_{j=s}^{r-1} A_j u_{s + r - 1 - j} = \sum_{j=s}^{r-1} A_j U f_{s + r - 1 - j}.
\]
Applying \( U^{-1} \), we get
\[
U^{-1} W_s = \sum_{j=s}^{r-1} M_{a_j} f_{s + r - 1 - j} = w_s.
\]
Hence, for every \( n \geq r \),
\[
f_n = \sum_{s=0}^{r-1} \rho(n - s, r; \cdot) U^{-1} W_s.
\]
Applying \( U \) to both sides yields
\[
u_n =Uf_n=\sum_{s=0}^{r-1} U  \rho(n - s, r; \cdot)  U^{-1}W_s  = \sum_{s=0}^{r-1} \rho(n - s, r; \mathbf{A}) \, W_s,
\]
where the operator \(\rho(n - s, r; \mathbf{A})\) is defined by functional calculus as
\begin{align*}
\rho(n - s, r; \mathbf{A}) 
   &:= U \rho(n - s, r; \cdot) U^{-1} \\
   &= \rho(n, r; \cdot) \\
   &= \sum_{\substack{\mathbf{k} \in \mathbb{Z}_+^r \\ \langle \mathbf{k}, \mathbf{d} \rangle = n - r}} 
      \binom{|\mathbf{k}|}{\mathbf{k}} \, 
      U a_0^{k_0} a_1^{k_1} \cdots a_{r-1}^{k_{r-1}} U^{-1} \\
   &= \sum_{\substack{\mathbf{k} \in \mathbb{Z}_+^r \\ \langle \mathbf{k}, \mathbf{d} \rangle = n - r}} 
      \binom{|\mathbf{k}|}{\mathbf{k}} \, \mathbf{A}^{\mathbf{k}} .
\end{align*}
and
\(
\rho(r, r; \mathbf{A}) =I_{\mathcal{H}}, \, 
\rho(m, r; \mathbf{A})=0_{\mathcal{H}},  \text{ for } m < r.\)
This concludes the proof.

\end{proof}
As a corollary, we obtain the following result, which extends~\cite[Proposition~2.1]{taher2001linear} 
to the setting of operator algebras on infinite-dimensional Hilbert spaces.

\begin{thm}
Let \((T_n)_{n \geq 0} \subseteq \mathcal{B}(\mathcal{H})\) be a sequence of bounded operators satisfying the operator-valued linear recurrence relation \eqref{bin:ope} of order \(r\).
Then, for all \(n \geq r\), the following explicit expression holds:
\[
T_n = \sum_{s=0}^{r-1} \rho(n - s, r; \mathbf{A}) \, W_s,
\]
 where
 \begin{itemize}
     \item 
 The operator coefficients \(\rho(m, r; \mathbf{A})\)  are defined as in \eqref{exp1} and \eqref{exp2}.
 \item The operators \(W_s \in \mathbf{B}(\mathcal{H})\) are defined by
  \[
  W_s := \sum_{j=s}^{r-1} A_j \, T_{s + r - 1 - j}, \quad \text{for } s = 0, \ldots, r-1,
  \]
 \end{itemize}
\end{thm}
\begin{proof}
    For each \(h \in \mathcal{H}\), we apply Theorem~\ref{mainr} to the sequence \((u_n)_{n \geq 0}\) defined by \(u_n := T_n h\).
\end{proof}

\bigskip
\subsection{Computation of powers of the operator-valued companion matrix}
Let \( (u_n)_{n\ge 0} \subseteq \mathcal{H} \) satisfy the recurrence~\eqref{bin:vect}.  
To compute the powers of the operator-valued companion matrix \(\mathbf{B}\), which is the block companion matrix given in \eqref{cm}, recall the state vector
\[
Y_n := 
\begin{pmatrix}
u_{n+r-1} \\[1mm]
u_{n+r-2} \\[1mm]
\vdots \\[1mm]
u_n
\end{pmatrix} \in \mathcal{H}^{(r)},
\]
satisfying
\[
Y_n = \mathbf{B}^n \, Y_0, \quad n \in \mathbb{N}.
\]
When the operators \(A_0, \dots, A_{r-1}\) commute, Theorem \ref{mainr} provides an explicit combinatorial formula for the solution of the system:
\[
u_n = \sum_{s=0}^{r-1} \rho(n - s, r; \mathbf{A}) \, W_s,
\]
where
\[
W_s := \sum_{j=s}^{r-1} A_j \, u_{s + r - 1 - j}, \quad s=0,\dots,r-1.
\]

\begin{thm}[Entries of $\mathbf{B}^n$]
Under the previous notations, for all $n \ge 0$, the entries of $\mathbf{B}^n = (B^{(n)}_{i,k})_{0 \le i,k \le r-1}$ are given by
\[
B^{(n)}_{i,k} = \sum_{s=0}^{r-1} \rho(n+i-s, r; \mathbf{A}) \, C_{s,k},
\]
where 
\[
C_{s,k} =
\begin{cases}
A_{s + k}, & 0 \le s + k \le r-1, \\
0, & \text{otherwise}.
\end{cases}
\]
\end{thm}

\begin{proof}
The general solution of the recurrence is
\[
u_n = \sum_{s=0}^{r-1} \rho(n-s, r; \mathbf{A}) \, W_s.
\]
Setting \(k=j-s\) gives \(W_s =\displaystyle \sum_{k=0}^{r-1} C_{s,k} \, u_{r-1-k}\) with
\[
C_{s,k} =
\begin{cases}
A_{s + k}, & 0 \le s + k \le r-1, \\
0, & \text{otherwise}.
\end{cases}
\]
Then, for \(i=0,\dots,r-1\),
\[
u_{n+i} = \sum_{k=0}^{r-1} \Big( \sum_{s=0}^{r-1} \rho(n+i-s, r; \mathbf{A}) \, C_{s,k} \Big) u_k,
\]
which yields the entries of $\mathbf{B}^n$ as
\[
B^{(n)}_{i,k} = \sum_{s=0}^{r-1} \rho(n+i-s, r; \mathbf{A}) \, C_{s,k}.
\]
\end{proof}

\begin{rem}
This result generalizes both versions: the matrix case \cite[Theorem 4.1]{benkhaldoun2021periodic} for constant matrix coefficients, and the scalar case given by Chen-Louck’s Theorem \cite[Theorem 3.1]{chen1996companion}.
\end{rem}

\section{The Case of Scalar Coefficients}
We now focus on the  recurrence relations given respectively by \eqref{bin:vect}  and \eqref{bin:ope} where the coefficients are scalar multiples of the identity, i.e.,
\[
A_k = a_k I_\mathcal{H}, \quad a_k \in \mathbb{R}, \quad k = 0, \dots, r-1,
\]
so that the recurrence takes the form
\begin{equation}\label{bs}
u_{n+r} = a_0 u_n + a_1 u_{n+1} + \cdots + a_{r-1} u_{n+r-1}, \quad \forall n \ge 0.
\end{equation}
\begin{equation}\label{bo}
T_{n+r} = a_0 T_n + a_1 T_{n+1} + \cdots + a_{r-1} T_{n+r-1}, \quad \forall n \geq 0,
\end{equation}
In this case, if $\lambda_1, \dots, \lambda_s $ are the distinct roots of the associated characteristic polynomial
\(
P(X) = X^r - a_{r-1} X^{r-1} - \cdots - a_0,
\)
with respective multiplicities $m_1, \dots, m_s$, we can write a general Binet-type formula for the sequence $(u_n)_{n \ge 0}$, analogous to the classical scalar case, as follows:

\begin{thm}[Explicit Binet Formula for Vector-Valued Recurrence]\label{bin:sca}
Let \( (u_n)_{n\ge 0} \subseteq \mathcal{H} \) satisfy the scalar-coefficient recurrence~\eqref{bs}. \ Then there exist unique vectors \( v_{i,j} \in \mathcal{H} \), indexed by \( 1 \le i \le s \) and \( 0 \le j \le m_i - 1 \), such that
\[
u_n = \sum_{i=1}^s \sum_{j=0}^{m_i - 1} v_{i,j} \, n^j \, \lambda_i^n,
\]
where the vectors $v_{i,j}$ are determined by the initial values $u_0, \dots, u_{r-1}$.
\end{thm}

\begin{proof}
For every nonzero $x \in \mathcal{H}$, the scalar sequence $(\langle u_n, x \rangle)_{n\ge 0}$ satisfies the scalar recurrence relation associated with \eqref{bin:vect}. \ Then there exist unique scalars $\beta_{i,j}(x) \in \mathbb{C}$, indexed by $1 \le i \le s$ and $0 \le j \le m_i - 1$, such that
\[
\langle u_n, x \rangle = \sum_{i=1}^s \sum_{j=0}^{m_i - 1} \beta_{i,j}(x) \, n^j \, \lambda_i^n,
\]
By uniqueness, the mapping $\beta_{i,j} : \mathcal{H} \to \mathbb{C}$ is a linear form, then from the Riesz Representation Theorem, there exists a unique vector $v_{i,j} \in \mathcal{H}$ such that
$$
\beta_{i,j}(x) = \displaystyle\langle v_{i,j}, x \rangle \quad (\text{for all } x \in \mathcal{H}).
$$
In particular, for all $x \in \mathcal{H}$, we have
\[
\langle u_n, x \rangle = \displaystyle\left\langle \sum_{i=1}^s \sum_{j=0}^{m_i - 1} v_{i,j} \, n^j \, \lambda_i^n, \, x \right\rangle,
\]
which implies the following identity in $\mathcal{H}$:
\[
u_n = \displaystyle\sum_{i=1}^s \sum_{j=0}^{m_i - 1} v_{i,j} \, n^j \, \lambda_i^n.
\]
\end{proof}
As a corollary, we obtain the following theorem for operators.
\begin{thm}[Explicit Binet Formula for an operator-valued recurrence] Let \( (T_n)_{n\ge 0} \subseteq \mathbf{B}(\mathcal{H}) \)  satisfy the scalar-coefficient recurrence~\eqref{bo}. \ Then there exist unique operators \( S_{i,j} \in \mathcal{H} \), indexed by \( 1 \le i \le s \) and \( 0 \le j \le m_i - 1 \), such that
\[
T_n = \sum_{i=1}^s \sum_{j=0}^{m_i - 1} S_{i,j} \, n^j \, \lambda_i^n,
\]
where the operators \( S_{i,j} \in \mathcal{B}(\mathcal{H}) \) are uniquely determined by the initial values \( T_0, \dots, T_{r-1} \).
\end{thm}
\begin{proof}
Fix any \( h \in \mathcal{H} \), and define the sequence of vectors \( u_n := T_n(h) \in \mathcal{H} \). \  
Then, by Theorem~\ref{bin:sca}, there exist unique vectors \( \beta_{i,j}(h) \in \mathcal{H} \) (for \( 1 \leq i \leq s \), \( 0 \leq j \leq m_i - 1 \)) such that
\(
T_n(h) = \displaystyle\sum_{i=1}^s \sum_{j=0}^{m_i - 1} \beta_{i,j}(h) \, n^j \, \lambda_i^n.
\)
For each pair \((i,j)\), define the operator \( S_{i,j} : \mathcal{H} \to \mathcal{H} \) by
\(
S_{i,j}(h) := \beta_{i,j}(h).
\)
By uniqueness, each \( S_{i,j} \) is a bounded linear operator, since it arises from a linear combination of the bounded operators \( T_0, \dots, T_{r-1} \) and the roots \(\lambda_i\).
\end{proof}
Under the assumptions of the previous theorem, if the characteristic polynomial has simple roots 
\( \{\lambda_1, \dots, \lambda_r\} \), then:

\begin{cor}
The sequence \( (T_n)_{n \geq 0} \) given by \eqref{bo} consists of the moments of a finitely atomic 
representing operator measure
\(
    E = \displaystyle\sum_{i=1}^r S_i \, \delta_{\lambda_i},
\)
where each \( S_i \in \mathbf{B}(\mathcal{H}) \) and \( \delta_{\lambda_i} \) is the Dirac measure concentrated 
at the atom \( \lambda_i \).\\ That is,
\(
    T_n = \displaystyle\sum_{i=1}^r S_i \, \lambda_i^n 
    = \int_{\mathbb{R}} t^n \, dE(t), 
    \quad \forall n \ge 0.
\)
\end{cor}

\begin{rem}
    Under the previous notation, the sequence \( (T_n)_{n \geq 0} \) is an operator moment sequence 
on \( \{\lambda_1, \dots, \lambda_r\} \) (i.e., the operators \( S_i \) are positive on \( \mathcal{H} \)) 
if and only if the representing operator measure is positive. \ 
This is equivalent to the positive semidefiniteness of the local moment matrix 
$$
    \bigl( \langle T_{i+j}x, x \rangle \bigr)_{0 \leq i,j \leq r-1}\in\mathbf{M}_r(\mathbb{R}),
$$
for every $x \in \mathcal{H}$ (see \cite[Theorem~6.10]{curto2025local}).

\end{rem}

\bigskip
\subsection{Algebraic operators}
\subsubsection{Powers of algebraic operators}
Recall that a bounded linear operator \( T \in \mathcal{B}(\mathcal{H}) \) is said to be \emph{algebraic} if there exists a non-zero polynomial \( P \in \mathbb{C}[X] \) such that
\(
P(T) = 0_\mathcal{H}.
\)
Let \( T \) be an algebraic operator and let \( P \) be the associated monic minimal polynomial of degree \( r \), given by
\[
P(X) = X^r - a_{r-1} X^{r-1} - \cdots - a_1 X - a_0,
\]
with real coefficients \(a_0, \dots, a_{r-1} \in \mathbb{R}\). \ Then \(T\) satisfies the polynomial identity
\[
P(T)=0_\mathcal{H} \quad \Longleftrightarrow \quad 
T^{n+r} = a_{r-1} T^{n+r-1} + \cdots + a_1 T^{n+1} + a_0 T^n, \quad \forall n \ge 0.
\]
This identity induces an operator-valued linear recurrence relation of order \( r \) for the sequence of powers \( (T^n)_{n \ge 0} \), with operator coefficients \( A_k = a_k I_{\mathcal{H}} \) for \( k = 0, \dots, r-1 \). \ This recurrence can be directly identified with the scalar case by applying the scalar product.
\begin{cor}
Under the above assumptions, if the characteristic polynomial 
\[
P(X) = X^r - a_{r-1} X^{r-1} - \cdots - a_1 X - a_0
\]
has roots \(\lambda_i\) with multiplicities \(m_i\), then for every \(n \geq r\), the powers of \(T\) admit the two following representations:

\begin{align*}
T^n &= \sum_{s=0}^{r-1} \rho(n - s, r) \, W_s 
\quad \text{(Combinatorial representation)} \notag \\
&= \sum_{i=1}^s \sum_{j=0}^{m_i -1} C_{i,j} \, n^j \, \lambda_i^n 
\quad \text{(Binet-type representation)}
\end{align*}
where the scalar coefficients \(\rho(m, r)\) are defined as in \eqref{eq:main}, the operators \(W_s \in \mathcal{B}(\mathcal{H})\) are given by
\[
W_s := \sum_{j=s}^{r-1} a_j \, T^{s + r - 1 - j}, \quad s=0, \dots, r-1,
\]
and the operators \(C_{i,j} \in \mathcal{B}(\mathcal{H})\) depend only on the initial terms \(T^0, \dots, T^{r-1}\).
\end{cor}
\subsubsection{Exponentials of algebraic operators}

In order to compute the exponential of algebraic operators, we need the Binet expression and some tools from the theory of Bell polynomials. \ Recall that the \emph{Bell number} $B_n$ counts the number of partitions of a set with $n$ elements. \ Dobinski's formula provides an explicit representation of the Bell numbers:
\[
B_n = \frac{1}{e} \sum_{k=0}^{\infty} \frac{k^n}{k!}.
\] 
The \emph{Stirling number of the second kind} with parameters $j$ and $k$, denoted
\(
S(j,k),
\)
enumerates the number of partitions of a set with $j$ elements into $k$ disjoint, nonempty subsets. \ In particular, the Bell numbers can be expressed as
\[
B_n = \sum_{k=0}^n S(n,k).
\]
The numbers $S(n,k)$ are also called \emph{Stirling partition numbers}. \ More generally, the $n$-th \emph{Bell polynomial} is defined by
\begin{equation}\label{bp}
B_n(x) = \sum_{k=0}^n S(n,k) \, x^k.
\end{equation}
These numbers and polynomials have many remarkable properties and appear in several 
combinatorial identities. \ A comprehensive reference is \cite{comtet2012}. \ More recently, the authors in \cite{mezo2011} generalized the Bell numbers and polynomials to the so-called $r$-Bell numbers and polynomials, as follows:

\begin{thm}[Dobinski's formula for $r$-Bell numbers {\cite[Theorem 5.1]{mezo2011}}]\label{bn}
The $r$-Bell polynomials satisfy the identity
\[
B_{n,r}(x) = \frac{1}{e^x} \sum_{k=0}^{\infty} \frac{(k+r)^n}{k!}\, x^k.
\]
Consequently, the $r$-Bell numbers are given by
\[
B_{n,r} = \frac{1}{e} \sum_{k=0}^{\infty} \frac{(k+r)^n}{k!}.
\]
\end{thm}
\begin{prop}[Exponential of $T$]
Under the above assumptions, the exponential of $T$ can be expressed explicitly as
\[
e^T = \sum_{i=1}^s \sum_{j=0}^{m_i-1} C_{i,j} \, e^{\lambda_i} P_j(\lambda_i),
\]
where $P_j(\lambda)$ is the \emph{Bell polynomial}, given explicitly by \eqref{bp}.
\end{prop}

\begin{proof}
By definition, the exponential of $T$ is given by the series
\[
e^T := \sum_{n=0}^{\infty} \frac{T^n}{n!}.
\]
Substituting the Binet-type formula for $T^n$ yields
\[
e^T 
= \sum_{i=1}^s \sum_{j=0}^{m_i-1} C_{i,j} \sum_{n=0}^{\infty} \frac{n^j \lambda_i^n}{n!}.
\]
Thus, it remains to compute the scalar series
\(
\displaystyle\sum_{n=0}^{\infty} \frac{n^j \lambda^n}{n!}.
\)
This is a special case of Theorem~\ref{bn} with $r=0$: for each $j \geq 0$,
\[
\sum_{n=0}^{\infty} \frac{n^j \lambda^n}{n!} = e^\lambda P_j(\lambda).
\]
Hence, we obtain the explicit formula
\[
e^T = \sum_{i=1}^s \sum_{j=0}^{m_i-1} C_{i,j} \, e^{\lambda_i} P_j(\lambda_i),
\]
which completes the proof.
\end{proof}

\begin{rem}
These results illustrate the interplay between combinatorial techniques and operator theory, providing effective tools for analyzing operator sequences in Hilbert spaces. 
\end{rem}

\section{Continuous-Time Operator Recurrence}

We consider next an application to the  \emph{continuous-time analogue} of the discrete operator recurrence studied in this paper. Such systems are classical in the study of \emph{linear evolution equations in Hilbert spaces}. \ More precisely $u: \mathbb{R} \longrightarrow \mathcal{H}$, satisfying a \emph{linear operator differential equation of order $r$}:

\begin{equation}\label{cont-recurrence}
\frac{d^r}{dt^r} u(t) = A_0 u(t) + A_1 \frac{d}{dt} u(t) + \dots + A_{r-1} \frac{d^{r-1}}{dt^{r-1}} u(t),
\end{equation}
where $A_0, \dots, A_{r-1} \in \mathcal{B}(\mathcal{H})$ are bounded operator coefficients. \

As in the discrete case, we introduce the \emph{state vector} in decreasing order of derivatives:
\[
Y(t) :=
\begin{pmatrix}
\frac{d^{r-1}}{dt^{r-1}} u(t) \\[1mm]
\frac{d^{r-2}}{dt^{r-2}} u(t) \\[1mm]
\vdots \\[1mm]
u(t)
\end{pmatrix} \in \mathcal{H}^{(r)},
\]
and rewrite \eqref{cont-recurrence} as a \emph{first-order operator differential system}:
\begin{equation}\label{cont-system}
\frac{d}{dt} Y(t) = \mathbf{B} \, Y(t), \quad Y(0) = Y_0 \in \mathcal{H}^{(r)},
\end{equation}
where $\mathbf{B} \in \mathcal{B}(\mathcal{H}^{(r)})$ is the \emph{operator-valued companion matrix} given by \eqref{cm}.
Formally, the solution is given by the exponential of the operator matrix:
\[
Y(t) = e^{t \mathbf{B}} \, Y_0.
\]
In other words:
\(
u(t) = 
\begin{bmatrix} 
0 & 0 & \cdots & I_{\mathcal{H}} 
\end{bmatrix} 
\, e^{t \mathbf{B}} \, Y_0,
\)
where 
\(
\begin{bmatrix} 0 & \cdots & I_{\mathcal{H}} \end{bmatrix}
\)
is the projection from \(\mathcal{H}^{(r)}\) onto the last component. \ From this formulation naturally raises the following  question.

\begin{quest}
How can we extend the combinatorial techniques and explicit formulas for powers of the companion matrix, developed in the discrete-time commuting case, to the continuous-time operator setting? In particular, can we obtain a closed-form expression for $e^{t \mathbf{B}}$\,?
\end{quest}

\section{Declarations}
\subsection{Funding} 
The first-named author was partially supported by NSF Grant DMS-2247167. \ The last-named author was partially supported by the Arab Fund Foundation Fellowship Program, The Distinguished Scholar Award-File 1026.

\subsection{Conflicts of interest/competing interests} 

{\bf Non-financial interests}: \ None.

\subsection{Data availability.}
All data generated or analyzed during this study are included in this article.

\end{document}